\documentclass[lebTerpaper,10pt]{amsart}
\usepackage{amsmath,amssymb,amsxtra,amsthm, amstext, amscd,amsfonts,fancyhdr, hyperref,enumerate,dirtytalk}
\usepackage[OT2,T1]{fontenc}
\DeclareSymbolFont{cyrlebTers}{OT2}{wncyr}{m}{n}
\DeclareMathSymbol{\Sha}{\mathalpha}{cyrlebTers}{"58}

\newcommand{\bC}{{\mathbb{C}}}

\newcommand{\bQ}{{\mathbb{Q}}}
\newcommand{\bR}{{\mathbb{R}}}

\newcommand{\bZ}{{\mathbb{Z}}}

\newcommand{\C}{{\mathcal{C}}}

\newcommand{\F}{{\mathcal{F}}}
\newcommand{\G}{{\mathcal{G}}}

\newcommand{\R}{{\mathcal{R}}}

\newcommand{\fB}{\mathfrak{B}}

\newcommand{\GL}{\operatorname{GL}}

\newcommand{\ol}{\overline}

\newcommand{\upchi}{{\raise.35ex\hbox{$\chi$}}}

\makeatletter
\@namedef{subjclassname@2010}{%
	\textup{2010} Mathematics Subject Classification}
\makeatother

\newtheorem{theorem}{Theorem}[section]

\newtheorem{proposition}[theorem]{Proposition}
\newtheorem{lemma}[theorem]{Lemma}

\theoremstyle{definition}
\newtheorem{definition}[theorem]{Definition}

\numberwithin{equation}{section}

\begin{document}
	
	\title{Sextic polynomials in two variables do not represent all non-negative integers}
	
	\author{Stanley Yao Xiao}
	\address{Department of Mathematics and Statistics \\
		University of Northern British Columbia \\
		3333 University Way \\
		Prince George, British Columbia, Canada \\  V2N 4Z9}
	\email{StanleyYao.Xiao@unbc.ca}
	\indent

	
	\begin{abstract} In this paper we prove that for all degree $6$ polynomials with rational coefficients that $F(\bZ^2) \ne \bZ_{\geq 0}$. The answers a question of B.~Poonen and J.~S.~Lew in the degree 6 case. This work builds on previous work with S.~Yamagishi but does not subsume it. 
		\end{abstract}
	
	\maketitle
	
	\section{Introduction}
	
	The study of representation of integers by bivariate polynomials go back to ancient times, at least to the time of Bhramagupta who proved in the 7th century that the product of two numbers representable as a sum of two squares is again a sum of two squares. Fermat had stated, without proof, that the odd primes which are sums of two squares are precisely those which are congruent to $1$ modulo $4$. This assertion was proved for the first time by Euler nearly a century later. \\
	
	Legendre, Lagrange, and Gauss investigated sums of more than two squares. Legendre proved that every number not of the form $4^m (8n+7)$ can be written as a sum of three squares, and Lagrange famously proved that every natural number is the sum of four integer squares. Gauss proved that every natural number is the sum of three triangular numbers. \\
	
	Motivated by these classical theorems, B.~Poonen asked on the popular website MathOverflow \cite{Poon} whether there exists a polynomial $F$ in two variables satisfying $F(\bZ^2) = \bZ_{\geq 0}$. It turns out that this question was asked much earlier by J.~S.~Lew \cite{Lew}. \\
	
	In \cite{XY} we addressed the Lew--Poonen question in the case of degree $4$ polynomials in two variables, giving a negative answer. The goal of the present paper is to show that the answer for degree six polynomials is also negative. Clearly, to answer the Ponnen-Lew question it suffices to show that for all polynomials $F \in \bZ[x,y]$ and integers $D_1, D_2$ with $D_1 > 0$, one has $F(\bZ^2) \ne D_1 \bZ + D_2$. To simplify the exposition, let us give the following definition: 
	
	\begin{definition} We say that $F(x,y) \in \bZ[x,y]$ is \emph{arithmetically complete} if there exists a positive integer $D_1$ and an integer $D_2$ such that $F(\bZ^2) = D_1 \bZ_{\geq 0} + D_2$. 
	\end{definition}
	
	Our main theorem is the following:
	
	\begin{theorem} \label{MT} There is no degree 6 polynomial defined over the integers in two variables which is arithmetically complete. 
	\end{theorem}
	
	As in \cite{XY}, for a subset $\fB \subseteq \bZ$ we define 
	\[\R_F(\fB) = \{n \in \fB : \exists (x,y) \in \bZ^2 \text{ s.t. } F(x,y) = n\}.\]
	R.~Stanley had asked about the asymptotic behaviour of the size of the set $\R_F([N,2N))$ when $F(\bZ^2) \subseteq \bZ_{\geq 0}$. In particular he asked whether the upper bound 
	\begin{equation} \label{stanley} \# \R_F([N, 2N)) = O_F \left(\frac{N}{\sqrt{\log N}} \right)\end{equation} 
	always holds. In \cite{XY} we answered Stanley's question in the affirmative for degree 4 polynomials. Our proof of Theorem \ref{MT} will also answer Stanley's question in the affirmative for degree six polynomials.
	
	\begin{theorem} \label{MT2} For all degree six polynomials $F$ satisfying $F(\bZ^2) \subseteq \bZ_{\geq 0}$ we have 
	\[\# \R_F([N,2N)) = O_F \left(\frac{N}{\sqrt{\log N}} \right).\]
	\end{theorem} 
	
In the case of binary forms, the question of estimating $R_F(N) = \# \R_F([-N,N])$ has been the subject of intense study, going back to the important work of Mahler \cite{Mah} and Erd\"{o}s--Mahler \cite{EM}. In \cite{SX}, C.~L.~Stewart and I proved that whenever $F$ has non-zero discriminant and $\deg F \geq 3$ there exists a positive number $C(F)$ such that 
\[R_F(N) \sim C(F) N^{\frac{2}{d}}.\]

In the case of degree 2, the seminal work of Landau \cite{Lan} and Bernays \cite{Ber} proved that whenever $F$ is an irreducible binary quadratic form that there exists a positive number $C(F)$ such that 
\[R_F(N) \sim C(F) \frac{N}{\sqrt{\log N}}.\]
We will now discuss the approach taken to prove Theorems \ref{MT} and \ref{MT2}. 
	
	\subsection{Outline of proof}
	
	The plan of attack in the present work is very similar to \cite{XY}, but the exposition is different. We will again consider the decomposition of a generic sextic polynomial into homogeneous components, namely
	\[F(x,y) = \sum_{j=1}^6 F_j(x,y),\]
	with each $F_j(x,y) \in \bZ[x,y]$ and homogeneous of degree $j$. Further, in \cite{XY} we essentially showed that all polynomials in two variables fail to be arithmetically complete for one of two reasons. To make this precise, we will make two definitions: 
	
	\begin{definition} We say that a polynomial $F \in \bZ[x,y]$ is \emph{arithmetically positive} if there exists an integer $D$ such that $F(\bZ^2) \subseteq \bZ_{\geq D}$. 
	\end{definition}
	
	\begin{definition} We say that a polynomial $F \in \bZ[x,y]$ is \emph{dearth} if we have the asymptotic relation 
	\[\# \R_F([-N,N]) = o_F(N).\]
	\end{definition}
	
	We note immediately that $F$ is arithmetically complete only if it is arithmetically positive and it is not dearth. \\
	
	In \cite{XY}, we showed that every quartic polynomial in two variables either fails to be arithmetically positive, meaning that $\inf_{(x,y) \in \bZ^2} F(x,y) = -\infty$, or $F$ is dearth. Further, in \cite{XY}, it is shown that the most difficult case corresponds to those $F$ (of degree $4$) taking the shape 
	\[F(x,y) = H(x,y)^{2r} G_4(x,y) + H(x,y) G_3(x,y) + F_2(x,y) + F_1(x,y),\]
	the so-called \emph{quadratically composed} polynomials. Here our situation is quite different. Indeed, the difficult case in the degree six situation is quite distinct from the one encountered in \cite{XY} and requires a different set of tools and ideas to handle. We emphasize that the ideas introduced in this paper do not subsume the ones in \cite{XY}. The main new ingredient, introduced to us by J.~Rouse \cite{XQ}, is Lemma \ref{Rouse}.  \\
	
	Our proof will be achieved through the following propositions: 
	
	\begin{proposition} \label{MP0} Let $F \in \bZ[x,y]$ be a degree six polynomial. Suppose that $F_6$ is square-free as a polynomial over an algebraic closure of $\bQ$. Then $F$ is arithmetically complete. 
	\end{proposition}
	
	\begin{proposition} \label{MP1} Let $F \in \bZ[x,y]$ be a degree six polynomial. Suppose that $F_6$ is divisible by the square of a binary form but not a cube over an algebraic closure of $\bQ$. Then $F$ is not arithmetically complete. 
	\end{proposition}
	
	\begin{proposition} \label{MP2} Let $F \in \bZ[x,y]$ be a degree six polynomial. Suppose that $F_6$ is divisible by the fourth power of a binary form but not a fifth power over an algebraic closure of $\bQ$. Then $F$ is not arithmetically complete. 
	\end{proposition}
	
	\begin{proposition} \label{MP3} Let $F \in \bZ[x,y]$ be a degree six polynomial. Suppose that $F_6$ a perfect 6th power over an algebraic closure of $\bQ$. Then $F$ is not arithmetically complete. 
	\end{proposition}
	
	Plainly, Propositions \ref{MP0} to \ref{MP3} exhaust all possible degree six polynomials in two variables defined over the integers, and thus implies Theorem \ref{MT}. \\
	
	Propositions \ref{MP0} and \ref{MP2} are essentially elementary, with Propositions \ref{MP1} and \ref{MP3} being more difficult to prove. The difficulty in the proofs of Propositions \ref{MP1} and \ref{MP3} is dealing with a type of polynomial similar to an example given by T.~Tao in \cite{Poon}. In particular, he stated that \footnote{``As such, a polynomial such as $f(x,y) = (x^2 - y^3 - y)^4 - y + C$ for some large constant $C$ already looks very tough to analyse." - T.~Tao, \url{https://mathoverflow.net/questions/9731/polynomial-representing-all-nonnegative-integers} } the polynomial
	\begin{equation} f(x,y) = (x^2 - y^3 - y)^4 - y + C, C \in \bZ
	\end{equation}
	looks difficult to analyse. This is similar to the most difficult case that needs to be dealt with for the proof of Proposition \ref{MP3}. Such polynomials take the form
	\begin{equation} \label{Taoshape} F(x,y) = a(y^2 - x^3 - Ax - B)^2 + g(x,y)
	\end{equation}
	where $\deg g \leq 2$. \\
	
	The main insight needed to tackle this difficult case differs depending on whether $A = 0$ or $A \ne 0$. In the former case, a Pell-type identity due to L.~V.~Danilov \cite{Dan} addressing Hall's conjecture is enough to show that such polynomials cannot be arithmetically complete. When $A \ne 0$, a construction due to J.~Rouse \cite{XQ} allows us to overcome this issue; see the proof of Lemma \ref{rouse}. 
	
	\section{Preliminary lemmas and proof of Proposition \ref{MP0}} 
	
We recall the following lemmata from \cite{XY}: 

\begin{lemma} \label{single}
Suppose there exists a polynomial $\mathfrak{F} \in \bQ[t]$ of degree $d \geq 2$ and $G \in \bQ[x,y]$ such that $F = \mathfrak{F} \circ G$. Then
\[\# \R_F([N, 2N]) \ll N^{\frac{1}{d}}.\]
\end{lemma}

\begin{lemma} \label{posdef lem}
Suppose $F_6$ is  positive definite. Then
$$
\# \mathcal{R}_F([N, 2N]) \ll N^{\frac{1}{3}}.
$$
\end{lemma}
	
	\begin{proof}
Since $F_6$ is positive definite, there exists a  constant $c > 0$ such that
	\[F_6(x,y) \geq c \max\{|x|, |y|\}^6  \]
	for all $(x,y) \in \bZ^2$.
It is clear that
$$
F(x,y) =  F_6 (x,y) + O( \max\{|x|, |y|\}^5  ).
$$
Let $N$ be a sufficiently large positive integer.
Then we obtain
\begin{eqnarray}
\# \mathcal{R}_F([N, 2N])
\notag
&=&
\# \{(x,y) \in \bZ^2 : N \leq F(x,y) \leq 2 N\}
\\
\notag
&\leq& \# \left\{(x,y) \in \bZ^2 : \max\{|x|, |y|\}  \ll N^{\frac{1}{6}}\right\}
\\
\notag
&\ll& N^{\frac{1}{3}}.
\end{eqnarray}

\end{proof}

Thus, Proposition \ref{MP0} is proved in the case when $F_6$ is positive definite. The indefinite case follows from another lemma in \cite{XY}: 

\begin{lemma} \label{indef lem} Suppose $F_{6}$ is negative semi-definite or indefinite. Then
$$
\inf_{(x,y) \in \bZ^2} F(x, y) = - \infty.
$$
\end{lemma}
\begin{proof}
Since $F_{6}$ is negative definite or indefinite, there exists $(x_0, y_0) \in \bZ^2$ such that
$$
F_{6}(x_0, y_0) < 0.
$$
It is clear that
$$
F(N x_0, N y_0) = N^6 F_6(x_0, y_0) + O(N^5),
$$
and the result follows.
\end{proof}

The proof of Proposition \ref{MP0} follows from the observation that a square-free sextic form $F_6$ is either positive definite, negative definite, or indefinite.

\section{Proof of Proposition \ref{MP1}}

Next we address the problem when $F_6$ is divisible by the square of a binary form but not a $4$-th power, say $F_6(x,y) = f(x,y)^2 G_6(x,y)$ with $f$ irreducible and $G_6$ co-prime to $f$. Note that $\deg f \leq 6/2 = 3$. We may further suppose, by Lemmas \ref{posdef lem} and \ref{indef lem}, that $F_6$ is positive semi-definite but not definite. \\

We will require the following adaptation of Lemma 3.1 from \cite{XY}: 

\begin{lemma} \label{gcd}
Suppose $F_{6}$ is positive semi-definite  and $\gcd(F_6, F_{5}) = 1$.
Then
$$
\inf_{(x,y) \in \bZ^2} F(x, y) = - \infty.
$$
\end{lemma}
\begin{proof}
We begin by noting that the hypotheses of the statement ensure that $F_5$ is not the zero polynomial.
Let $(\cos (\xi), \sin(\xi))$ be a real zero of $F_6$, which exists by the assumption that $F_6$ is positive semi-definite but not definite. We then write
$$
\mathfrak{L}(x,y) = \sin (\xi) x - \cos (\xi) y
$$
and
$$
F_6(x, y) = \mathfrak{L}(x, y)^{2 r} \mathfrak{F}_{6} (x, y),
$$
where $r \in \mathbb{N}$ and $\mathfrak{L} \nmid \mathfrak{F}_{6}$. In particular, since $\gcd(F_6, F_{5}) = 1$, we have $\mathfrak{L} \nmid F_{5}$.
By Dirichlet's theorem on Diophantine approximation,
there exist infinitely many $(u, v) \in \bZ^2$ satisfying
$$
|\sin (\xi) u - \cos (\xi) v| \ll \frac{1}{\sqrt{ u^2 + v^2 }}.
$$
For these $(u, v)$, we have
$$
F_{6}(u, v) = \mathfrak{L}(u, v)^{2 r} \mathfrak{F}_{6} (u, v)  \ll ( u^2 + v^2 )^{- r} (u^2 + v^2 )^{2 -  r} \ll (u^2 + v^2)^{2 - 2 r} \ll 1,
$$
$$
| F_{5}(u, v) |  \gg  (u^2 + v^2)^{ \frac{5}{2} }
\quad
\textnormal{and}
\quad
|F_{4}(u, v) + \cdots + F_1(u, v)| \ll \left(u^2 + v^2 \right)^2.
$$
Therefore, we obtain
$$
F (u, v) = F_{5}(u, v) + O\left(u^4 + v^4 \right).
$$
By replacing $(u,v)$ with $(-u, - v)$ if necessary, it follows that
$$
F (u, v) < 0
\quad
\textnormal{and}
\quad
|F(u, v)| \gg (u^2 + v^2)^{ \frac{5}{2}},
$$
provided $\sqrt{u^2 + v^2}$ is sufficiently large, and the result follows.
\end{proof}	

Next we shall need:

\begin{lemma} \label{bdlem} Let $\delta > 0$, and for a subset $\fB \subseteq \bZ_{\geq 0}$ put 
\[\R_F^{(\delta)}(\fB) = \left\{n \in \fB : \exists (x,y) \in \bZ^2 \text{ s.t. } F(x,y) = n \text{ and } F(x,y) \gg_\delta \max\{|x|, |y|\}^{1 + \delta}\right\}. \]
Then there exists $\kappa = \kappa(\delta) > 0$ such that
\[\# \R_F^{(\delta)} ([N,2N)) = o_\delta \left(N^{1 - \kappa(\delta)} \right).\]
\end{lemma} 

\begin{proof} Suppose that $F(x,y) = n$ with $N \leq n < 2N$ and $F(x,y) \gg_\delta \max\{|x|, |y|\}^{1 + \delta}$. It then follows that 
\[\max\{|x|, |y|\} \ll_\delta N^{\frac{1}{1+\delta}}.\]
It thus follows that 
\[\# \R_F^{(\delta)}([N, 2N)) \ll_\delta \#  \left \{(x,y) \in \bZ^2 : \max\{|x|, |y|\} \ll_\delta N^{\frac{1}{1+\delta}} \right\} = O \left(N^{\frac{2}{1+\delta}} \right).\]
If $\frac{2}{1 + \delta} < 1$, then we may take $\kappa = 2/(1+\delta)$. Thus we may assume that $\delta \leq 1$, and we consider the set 
\[\R_F^{(\delta, \ast)}([N,2N)) = \]
\[  \left\{n \in [N, 2N) : \exists (x,y) \in \bZ^2 \text{ s.t. } F(x,y) = n \text{ and } \max\{|x|, |y|\}^{1 + \delta} \ll_\delta F(x,y) \ll \max\{|x|, |y|\}^2\right\}.\]
The upper bound condition and the fact that $\deg F = 6$ implies that $F_6(x,y) \ll \max\{|x|, |y|\}^2$. In particular, for a real linear factor $\ell(x,y)$ of $F_6$ we must have $|\ell(x,y)| = O(1)$. Thus, $y$ is determined up to a constant factor given $x$. Then the condition 
\[\max\{|x|, |y|\}^{1 + \delta} \ll_\delta N \]
implies that there are $O_\delta \left(N^{1/(1+\delta)} \right)$ choices for $x$. It follows that 
\[\# \R_F^{(\delta, \ast)}([N, 2N)) = O_\delta \left(N^{\frac{1}{1 + \delta}} \right)\]
and the claim follows.

\end{proof}

Let $\ell$ be a real linear factor of $f$. Over $\ol{\bQ}$, we then write 
\begin{equation} \label{F6L} F_6(x,y) = \ell(x,y)^2 \F_6(x,y),\end{equation}
where $\gcd(\ell, \F_6) = 1$. By Lemma \ref{gcd} we may suppose that $\ell | F_5$. Hence, if $|\ell(x,y)| \gg 1$, then $F(x,y) \gg F_6(x,y) \gg \max\{|x|, |y|\}^4$, so we are done by Lemma \ref{bdlem}. Therefore we may assume that $|\ell(x,y)| = O(1)$. We record this as a lemma: 

\begin{lemma} \label{bigL} Suppose $F_6$ takes the shape (\ref{F6L}) for some real linear form $\ell$. Then the contribution from those $x,y \in \bZ$ satisfying $|\ell(x,y)| \gg 1$ is negligible. 
\end{lemma}

We will require the following result, which appears as Lemma 5.1 in \cite{XY}: 

\begin{lemma} \label{quadpoly} Let $Q \in \bQ[x,y]$ be a quadratic polynomial. Let $Q_2$ be the degree $2$ homogeneous part of $Q$,
and suppose that $Q_2$ is a rational quadratic form with non-zero discriminant. Then there exist $q_1, q_2, q_3 \in \bQ$, depending only on $Q$, such that
\[ Q(x,y) = Q_2(x + q_1, y + q_2) + q_3.\]
\end{lemma}

Next we treat the separate cases when $F_6$ is divisible by the square of a cubic, quadratic, or linear form defined over $\bQ$. 

\subsection{When $F_6$ is divisible by $f^2$, $f$ irreducible cubic} In this case $F_6 = a f(x,y)^2$ for some rational number $a$. By Lemma \ref{gcd} we note that $f$ divides $F_5$. Put $F_5 = f(x,y) g_5(x,y)$. If $\gcd(f, F_4) = 1$ then we may find, using Dirichlet's approximation theorem, infinitely many integers $p,q$ with $\gcd(p,q) = 1, q \geq 1$ such that
\[F_6(p,q) = O(q^2), |F_5(p,q)| = O(q^3) \quad \text{and} \quad |F_4(p,q)| \gg q^4.\]
Hence $|F_6(x,y)| \gg |F_4(x,y)|$ for all $x,y \in \bZ$. It follows that either $F(x,y) \gg |F_4(x,y)| \gg \max\{|x|, |y|\}^4$ for all $x,y \in \bZ$, if $F_4$ is positive on the region $|f(x,y)| \ll \max\{|x|, |y|\}$, in which case we are done by Lemma \ref{bdlem}, or $F$ takes on arbitrarily large negative values and thus $F$ fails to be arithmetically positive. We may therefore assume that $f | F_4$ as well. It follows that we may write 
\[F(x,y) = a f(x,y)^2 + f(x,y) g_5(x,y) + f(x,y) g_4(x,y) + F_3(x,y) + F_2(x,y) + F_1(x,y).\]

Put $g(x,y) = g_5(x,y) + g_4(x,y)$ and complete the square to obtain
\begin{equation} F(x,y) = a(f(x,y) + g(x,y))^2 - a g(x,y)^2 + F_3(x,y) + F_2(x,y) + F_1(x,y).
\end{equation} 
By replacing $g$ with $g + m$ for some integer $m$ if necessary, we may assume that the curve $\C$ defined by the cubic polynomial 
\[f(x,y) + g(x,y) + m\]
has a rational point. Since $f$ is an unramified binary cubic form, the cubic curve $\C$ is non-singular. Therefore, applying a rational transformation we may transform $F$ to take the shape
\[\F(x,y) = a (y^2 - x^3 - Ax - B)^2 + \G(x,y),\]
with $\deg \G \leq 4$. This situation will then be handled by Lemma \ref{Rouse}.  

\subsection{When $F_6$ is divisible by $f^2$, $f$ irreducible quadratic} Note that we may assume $f^3 \nmid F_6$, since if $F_6 = a f(x,y)^3$ then $F_6$ is either positive definite or indefinite, and we would be done by Lemmas \ref{posdef lem} or \ref{indef lem}. \\

Thus we may assume $F_6(x,y) = f(x,y)^2 g(x,y)$ with $g$ a non-singular quadratic form satisfying $\gcd(f,g) = 1$. We may further suppose that $f$ is indefinite, as follows. Note that $g$ must be positive definite by Lemma \ref{indef lem}. If $f$ is also positive definite then $F_6$ is positive definite, and we are done by Lemma \ref{posdef lem}. By Lemma \ref{gcd}, it follows that $f | F_5$. Put $F_5(x,y) = f(x,y) h(x,y)$.\\

By applying a rational transformation, we may assume $f(x,y) = x^2 - ky^2$ with $k$ positive square-free integer. Putting $x = \sqrt{k} y + c$, with $c = O(1)$, then gives
\[F_6(x,y) = c^2 (2 \sqrt{k} y + c)^2 (y^2 g(\sqrt{k}, 1) + cy g^\prime(\sqrt{k},1) + O(1)),\]
\[F_5(x,y) = c (2 \sqrt{k} y + c) (y^3 h(\sqrt{k}, 1) + c y^2 h^\prime(\sqrt{k},1)), F_4(x,y) = y^4 F_4(\sqrt{k}, 1) + cy^3 F_4^\prime(\sqrt{k},1) + O(y^2),\]
\[F_3(x,y) = y^3 F_3(\sqrt{k}, 1) + O(y^2).\]
Grouping the terms together, we find that
\begin{align*} & F(x,y)  = y^4(4kc^2 g(\sqrt{k},1) + 2\sqrt{k} c h(\sqrt{k},1) + F_4(\sqrt{k},1)) \\
& + y^3 (c^3 (4k g^\prime(\sqrt{k}, 1) + 4 \sqrt{k} g(\sqrt{k},1) + c^2 (h(\sqrt{k},1) + 2 \sqrt{k} h^\prime(\sqrt{k},1) + c F_4^\prime(\sqrt{k},1) + F_3(\sqrt{k},1))  + O(y^2). 
\end{align*}
Put
\[v_k(z) = 4kg(\sqrt{k},1) z^2 + 2\sqrt{k} h(\sqrt{k},1) z + F_4(\sqrt{k}, 1).\]
If $v_k(z)$ is not a perfect square over $\bR$, then $F$ fails to be arithmetically complete. Hence we must have $v_k(z) = A_k (z - \beta)^2$ for $\beta \in \bR \cap \ol{\bQ}$. If the cubic polynomial
\[w_k(z) = (4 k g^\prime(\sqrt{k}, 1) + 4 \sqrt{k} g(\sqrt{k}, 1)) z^3 + (h (\sqrt{k}, 1) + 2 \sqrt{k} h^\prime(\sqrt{k},1)) z^2 + F_4^\prime(\sqrt{k}, 1) z + F_3(\sqrt{k}, 1)\]
does not have a root at $\beta$, then by applying Dirichlet's approximation theorem we see that there exist infinitely many $x,y \in \bZ$ satisfying $|F(x,y)| \asymp \max\{|x|, |y|\}^3$, and by indefiniteness, this shows that $\inf_{(x,y) \in \bZ^2} F(x,y) = -\infty$. We may therefore assume that $w_k(\beta) = 0$. We may then complete the square to write
\[F(x,y) = y^2 \left((y u_1(c) +  u_2(c))^2 + u_0(c) \right) + O(y),\]
where $u_1, u_2$ are linear and quadratic polynomials in $c$ and $u_0$ is at most a quartic polynomial in $c$. By construction, $u_1$ vanishes at $\beta$. We now use the observation that since $v_k(z)$ is defined over $\bQ(\sqrt{k})$ and it is a square, its root must lie in $\bQ(\sqrt{k})$. In particular, there exists a positive number $c_\beta$ such that 
\[|x - \sqrt{k} y - \beta| > \frac{c_\beta}{|y|}\]
for all $x,y \in \bZ$. This then implies that
\[|yu_1(c) + u_2(c)| \gg c_\beta^\prime > 0\]
for all $x,y \in \bZ$, and hence $F(x,y) \gg y^2$ for all $x,y \in \bZ$. We are then done by Lemma \ref{bdlem}.

\section{Proof of Proposition \ref{MP2}}

Next we address the problem when $F_6$ is divisible by a perfect 4-th power, but not a fifth. Then the divisor must in fact be defined over $\bQ$, and up to a $\GL_2(\bZ)$-change of variables, we may assume that the divisor is $x^4$. That is, we assume that
\[F_6(x,y) = x^4 f_6(x,y), \deg f_6 = 2.\]
Suppose that $x^2 \nmid F_5$. Now suppose that $|x| \asymp T^{\theta}, |y| \asymp T$ for $1/2 < \theta < 2/3$ we find that 
\[F_6(x,y) = O\left(T^{4 \theta +2}\right), \quad |F_5(x,y)| \gg T^{\theta + 4}, \quad \text{and} \quad |F_4(x,y)| = O \left(T^4 \right).\]
Then we see that 
\[|F(x,y)| \gg |F_5(x,y)| \gg T^{\theta + 4}.\]
Since $F_5$ is guaranteed to be indefinite, we see that $F$ is not arithmetically complete. \\

Thus we may suppose that $x^2 | F_5(x,y)$. This enables us to write 
\[F(x,y) = y^2(a_2 x^4 + a_1 x^2 y + a_0 y^2) + xy (b_2 x^4 + b_1 x^2 y + b_0 y^2) + x^2 (c_2 x^4 + c_1 x^2 y + c_0 y^2) + G_5(x,y).\]
When $|x| \ll |y|$, the first term is dominant unless 
\[|a_2 x^4 + a_1 x^2 y + a_0 y^2| = O (|xy^{-1}|) = O(1).\]

Further, this polynomial must be a perfect square, for similar reasons as discussed before. This then corresponds to a factorization of the form 
\[a_2 x^4 + a_1 x^2 y + a_0 y^2 = (\alpha_1 x^2 - \alpha_2 y)^2.\]
If $\alpha_2 = 0$ then we must have $|x| = O(1)$. For each of these values of $x$, say $\{x_1, \cdots, x_k\}$, we may consider the single-variable polynomials
\[F_j(y) = F(x_j, y)\]
and conclude from Lemma \ref{single} that $F$ is dearth. Thus we may assume that $\alpha_2 \ne 0$. Further, by replacing $y$ with $-y$ if necessary we may assume that $\alpha_1, \alpha_2$ are both positive. Then the polynomial 
\[g_3(x,y) = b_2 x^4 + b_1 x^2 y + b_0 y^2\]
must also vanish at the weighted linear factor $\alpha_1 x^2 - \alpha_2 y$, and this corresponds to the polynomial $b_2 x^4 + b_1 x^2 y + b_0 y^2$ being divisible by $\alpha_1 x^2 - \alpha_2 y $. This enables us to once again complete the square and write 
\begin{align*} F(x,y) & = A \left(y(\alpha_1 x^2 - \alpha_2 y) + x(\beta_1 x^2 - \beta_2 y)\right)^2 + x^2 (h_2 x^4 + h_1 x^2 y + h_0 y^2) + G_5(x,y) \\
& = A(x,y)^2 + B(x,y) + G_5(x,y),\end{align*}
say. We then write 
\[y = \alpha_1 \alpha_2^{-1} x^2 + s.\]
If $|s| \gg x^2$, then $F(x,y) \gg y^4 = \max\{|x|, |y|\}^4$, and we are done by Lemma \ref{bdlem}. Otherwise we must have $s = o(x^2)$. Then we have
\begin{align*} y (\alpha_1 x^2 - \alpha_2 y) + x (\beta_1 x^2 - \beta_2y) & = (\alpha_1 \alpha_2^{-1} x^2 + s)s + x \left(\frac{\beta_1 \alpha_2 - \beta_2 \alpha_1}{\alpha_2} x^2 - \beta_2 s \right) \\
& = x^2 \left(\frac{\alpha_1}{\alpha_2} s - \frac{\beta_2 \alpha_1 - \beta_1 \alpha_2}{\alpha_2} x \right) + s^2 - \beta_2 sx.
\end{align*} 
Now choosing 
\[s = \frac{\beta_2 \alpha_1 - \beta_1 \alpha_2}{\alpha_1}x \]
we see that it is possible to make 
\[\left(y(\alpha_1 x^2 - \alpha_2 y) + x(\beta_1 x^2 - \beta_2 y) \right)^2 = O(x^2).\]
Note that $B$, having no common component with $A$, must have constant sign on the region $y \sim \alpha_1 \alpha_2^{-1} x^2$. If it has positive sign then we see that $F(x,y) \gg x^6 \gg \max\{|x|, |y|\}^3$ on this region and we are done by Lemma \ref{bdlem}. If it has negative sign, then using the above construction we see that we can make $A(x,y)^2 = O(x^4)$ while $|B(x,y)| \gg x^6$, and so $F$ is negative and has order of magnitude $x^6$ infinitely often, and we see that $F$ fails to be arithmetically positive. \\

Therefore it must be the case that $B \equiv 0$. Further, the analysis above holds whenever $G_5$ has any terms exceeding $x^4$ in order of magnitude in the regime when $|\alpha_1 x^2 - \alpha_2 y|^2 = O(|xy^{-1}|)$ . Thus we may assume $G_5$ takes the shape
\[G_5(x,y) = (d_2 x^4 + d_1 x^2 y + d_0y^2) + h_1 x^3 + h_0 xy + e_2 x^2 + e_1 y + e_0 x.\] 
Now, using the substitution 
\[y = \frac{\alpha_1}{\alpha_2} x^2 + \frac{\beta_2 \alpha_1 - \beta_1 \alpha_2}{\alpha_1} x + t\]
we find that 
\[F(x,y) =  D_3 x^4 + D_2 x^2 t + D_1 xt + D_0 t^2 + E_3 x^3 + E_2 x^2 + E_1 x. \]
Therefore, upon a polynomial change of variables, $F$ becomes a function in degree $4$ in two variables, so we are done by the main theorem in \cite{XY}. 

\section{Proof of Proposition \ref{MP3}} 

Lastly, we deal with the case when $F_6(x,y) = f_6 x^6$, so that it is totally ramified. Similar to Lemma to the arguments in the previous section we conclude:

\begin{lemma} \label{x6x5} Suppose that $F_6(x,y) = f_6 x^6$ and $x^3 \nmid F_5$. Then $F$ is not arithmetically positive. 
\end{lemma}

\begin{proof} Choose $|x| \asymp T^\theta, |y| \asymp T$. Then the hypotheses of the lemma imply that 
\[F_6(x,y) \asymp T^{6 \theta}, |F_5(x,y)| \asymp T^{\theta + 4}, |F_4 + \cdots + F_1| = O \left(T^4 \right).\]
Choosing $\theta = 1/2$ shows that 
\[F(x,y) \asymp F_5(x,y).\]
Since $F_5$ is necessarily indefinite, we see that $F$ cannot be arithmetically positive. If $x^3 \nmid F_5$ then
\[F_6(x,y) \asymp T^{6 \theta}, |F_5(x,y)| \asymp T^{2 \theta + 3}, |F_4 + \cdots + F_1| = O(T^4).\]
Choosing $\theta = 2/3$ we see that 
\[|F(x,y)| \asymp |F_5(x,y)|,\]
and thus $F$ again fails to be arithmetically complete. This completes the proof. \end{proof}

Next we consider the case when $x^4 | F_5$.

\subsection{The case when $x^4$ divides $F_5$}

We show that if $x^4 | F_5$ and $x | F_4$ exactly, then $F$ is not arithmetically complete.

\begin{lemma} \label{degen} Suppose $F$ is such that $F_6(x,y) = a_6 x^6$, $x^4 | F_5$, and $x | F_4$ exactly. Then $F$ is not arithmetically complete.  
\end{lemma}
\begin{proof} Choose $|x| \asymp T^\theta$ and $|y| \asymp T$. We note that the hypothesis implies that $F_4$ is indefinite. Note that
\[F_6(x,y) \asymp T^{6 \theta}, |F_5(x,y)| = O \left(T^{4 \theta + 1}\right), |F_4(x,y)| \asymp T^{\theta + 3}, |F_3 + \cdots + F_1| = O(T^3).\]
Choosing $\theta = 1/6$ we find that 
\[F_6(x,y) \asymp T, |F_5(x,y)| = O \left(T^{\frac{5}{3}} \right), |F_4(x,y)| \gg T^{\frac{19}{6}}, |F_3 + \cdots + F_1| = O \left(T^3 \right),\]
so $|F(x,y)| \asymp |F_4(x,y)|$ and thus cannot be arithmetically complete. 
\end{proof}

Therefore, we may assume that $x^2 | F_4$. Then $F$ takes the form
\begin{equation} \label{F40} F(x,y) = u_3 x^6 + u_2 x^4 y + u_1 x^2 y^2 + u_0 y^3 + xy L_1(x^2, y) + x^2 L_2(x^2, y) + yL_3(x^2, y) + x L_4(x^2, y) + L_5(x^2,y) + vx
\end{equation}
where $L_1, \cdots, L_5$ are linear forms and $u_3, u_2, u_1, u_0, v \in \bZ$. We study the lead form 
\[\F(x,y) =  u_3 x^6 + u_2 x^4 y + u_1 x^2 y^2 + u_0 y^3.\]
Note that this is a weighted homogeneous polynomial and a binary cubic form in the variables $x^2, y$, say 
\[\F(x,y) = \G(x^2, y) \quad \text{with} \quad \G(m,n) = u_3 m^3 + u_2 m^2 n + u_1 mn^2 + u_0 n^3.\] 
In particular, $\F$ is most certainly indefinite. Thus there exist integers $c_1, c_2$ such that $\G(c_1^2, c_2) < 0$. By weighted homogeneity we have
\[\F((Nx)^2, N^2 y) = N^6 \F(x^2, y). \]
Moreover, each of the terms in (\ref{F40}) is weighted homogeneous. For $N$ sufficiently large we have 
\[\left \lvert xy L_1(x^2, y) + x^2 L_2(x^2, y) + yL_3(x^2, y) + x L_4(x^2, y) + L_5(x^2,y) + vx \right \rvert = O \left(N^5 \right).\] 
This shows that 
\[|F(N^2 x^2, N^2 y)| \asymp N^6 |F(c_1^2, c_2)|\]
for $N$ sufficiently large, and in particular 
\[\inf_{(x,y) \in \bZ^2} F(x,y) = -\infty.\]
This shows that $F$ cannot be arithmetically positive. \\

Therefore, we may assume that $x^3$ divides $F_5$ exactly. 

\subsection{When $x^3$ divides $F_5$ exactly} 

We may therefore assume that $x^3 | F_5$. This enables us to write $F$ in the form 

\begin{equation} \label{calF0}F(x,y) = a_2 x^6 + a_1 x^3 y^2 + a_0 y^4 + xy L_1(x^3, y^2) + x^2 L_2(x^3, y^2) + yL_3(x^3, y^2)  + x L_4(x^3, y^2) + G(x,y),\end{equation} 
where each of the $L_i$'s is a linear form with integer coefficients. Note that the weighted homogeneous polynomial
\begin{equation} \label{calF} \F(x,y) = a_2 x^6 + a_1 x^3 y^2 + a_0 y^4\end{equation}
is a binary form in $x^3$ and $y^2$, and thus splits over $\bC$. Similar to the proof of Proposition \ref{MP2}, we conclude the following:

\begin{lemma} \label{leadram2} If the polynomial $\F$ is not a perfect square over $\bR$, then $F$ is not arithmetically positive. 
\end{lemma}

Hence we must have
\[\F(x,y) = a (\alpha_2 x^3 - \alpha_1 y^2)^2.\]
with $\alpha_2, \alpha_1 \ne 0$. Replacing $x$ with $-x$ if necessary we may freely choose the sign of $\alpha_2$, so we may assume that $\alpha_1, \alpha_2 > 0$. \\

Setting $y$ to the nearest integer to $\sqrt{|\alpha_2 \alpha_1^{-1} x^3|}$, we can make $|\F(x,y)| = O(x^4)$ while 
\[|xyL_1(x^3, y^2)| \gg |x|^{11/2}\] 
unless $L_1(x,y)$ is proportional to $\alpha_2 x^3 - \alpha_1 y^2$. We come to the same conclusion regarding $L_2, L_3, L_4$ by comparing sizes. This restriction again allows us to write
\[F(x,y) = a^\prime \left((\alpha_2 x^3 - \alpha_1 y^2) + \beta_1 xy + \beta_2 x^2 + \beta_3 x + \beta_4 y\right)^2 + G(x,y).\]
We now seek a substitution of the form 
\[y = \gamma_3 x^{\frac{3}{2}} + \gamma_2 x + \gamma_1 x^{\frac{1}{2}} + \gamma_0\]
which makes the term 
\[(\alpha_2 x^3 - \alpha_1 y^2) + \beta_1 xy + \beta_2 x^2 + \beta_3 x + \beta_4 y\]
as small as possible. Expanding, we obtain
\begin{align*} & \alpha_2 x^3 - \alpha_1 y^2 \\
& = \alpha_2 x^3 - \alpha_1 \bigg(\gamma_3^2 x^3 - (\gamma_2^2 + 2 \gamma_3 \gamma_1) x^2 -( \gamma_1^2 + 2 \gamma_2 \gamma_0) x - 2 \gamma_3 \gamma_2 x^{\frac{5}{2}} - 2 (\gamma_1 \gamma_2  + \gamma_3 \gamma_0) x^{\frac{3}{2}}  - 2 \gamma_1 \gamma_0 x^{\frac{1}{2}} - \gamma_0^2 \bigg).
\end{align*}
Comparing the $x^3, x^{\frac{5}{2}}, x^2$ with $ \beta_1 xy + \beta_2 x^2 + \beta_3 x + \beta_4 y$ terms and setting them to zero, we obtain the equations 
\[\alpha_2 = \alpha_1 \gamma_3^2, -2 \alpha_1 \gamma_2 \gamma_3 = \beta_1 \gamma_3, -\alpha_1(\gamma_2^2 +2\gamma_3 \gamma_1) = \beta_1 \gamma_2 + \beta_2 \gamma_3.\]
The first equation gives 
\[\gamma_3 = \sqrt{\frac{\alpha_2}{\alpha_1}}\]
and the second gives 
\[\gamma_2 = -\frac{\beta_1}{2\alpha_1}.\]
Substituting these into the third equation gives
\[\gamma_1 = - \frac{\beta_1^2}{8 \sqrt{\alpha_1 \alpha_2}} + \frac{\beta_2}{2}.\]
We then leave $\gamma_0$ free to ensure that $y \in \bZ$. This choice of $y$ then ensures that 
\[\left \lvert \alpha_2 x^3 - \alpha_1 y^2 + \beta_1 xy + \beta_2 x^2 + \beta_3 x + \beta_4 y\right \rvert = O(|x|^{3/2}).\]
Now consider the cubic curve defined by the equation
\[\C :  \alpha_2 x^3 - \alpha_1 y^2 + \beta_1 xy + \beta_2 x^2 + \beta_3 x + \beta_4 y = 0.\]
There exists a $\delta > 0$ such that there are infinitely many integral points $(u,v)$ which are a distance $O(\max\{|u|, |v|\}^{-\delta})$ from $\C$, which then implies that there are infinitely many integers $u,v$ such that 
\[\left \lvert \alpha_2 x^3 - \alpha_1 y^2 + \beta_1 xy + \beta_2 x^2 + \beta_3 x + \beta_4 y\right \rvert = O(|x|^{3/2 - \delta}).\]
Thus, all surviving terms in $G$ must be $o(|x|^3)$. This implies that we can also absorb the $\alpha_2^\prime x^3 - \alpha_1^\prime y^2$ term in $G$ into the square, resulting in 
\[F(x,y) =  \left( \alpha_2 x^3 - \alpha_1 y^2 + \beta_1 xy + \beta_2 x^2 + \beta_3 x + \beta_4 y + \beta_5 \right )^2 + G^\prime(x,y),\]
where each monomial appearing in $G^\prime$ has order of magnitude $O(|x|^{5/2})$. \\

After a rational change of coordinates we may express $F$ as 
\begin{equation} \label{ECform} F(x,y) = a_1 \left( y^2 - x^3 - b_1 x - b_0\right)^2 + \G(x,y).\end{equation}

We then have the following lemma:

\begin{lemma} \label{Rouse} For $F$ given as in (\ref{ECform}), $F$ is not arithmetically complete. \end{lemma} 

\begin{proof} The claim is clear if $\G(x,y)$ is constant by Lemma \ref{single}, so we may assume that $\G$ is non-constant. If $\G(x,y)$ is positive on the region defined by $|y^2 - x^3 - b_1 x - b_0| = O(x^{1/3})$, then we are done by Lemma \ref{bdlem}. Otherwise, we may assume that $\G$ is negative and tending to $-\infty$ on this region. \\ 

The key to this argument is to show that 
\[|y^2 - x^3 - b_1 x - b_0| = O(x^{1/3})\]
infinitely often for $x,y \in \bZ$. For $b_1 \ne 0$ we have the following construction due to Jeremy Rouse, posted on MathOverflow \cite{XQ}. Consider the elliptic curve defined by 
\begin{equation} \label{rouse} E_r : y^2 = x^3 + b_1 x + r^2 b_1^2,\end{equation}
where $r$ is a parameter we choose later. Note that the curve $E_r$ in (\ref{rouse}) has the obvious point $P = (0, rb_1)$. Using the addition law on $E_r$, we compute $3P$ to be 
\[3P = (64b_1^2 r^6 + 8b_1 r^2, 512 b_1^3 r^9 + 96 b_1^2 r^5 + 3b_1 r) = (x_r, y_r).\]
Since the power of $r$ exceeds that of the corresponding power of $b_1$ in each monomial, we can choose infinitely many $r \in \bZ$ to ensure that $3P \in \bZ^2$. Since $3P \in E_r$ it follows that 
\[|y_r^2 - x_r^3 - b_1 x_r - b_0| = |b_1^2 r^2 - b_0| = O(r^2) = O(x^{1/3}).\] 
It follows that there are infinitely many integers $x$ such that $F(x,y)$ is negative. \\

If $b_1 = 0$ then the above argument will not work, since for any $r$ the $x$-coordinate of $3P$ will remain zero. However we have a different construction due to Danilov (see also \cite{Elk}), which gives that there are infinitely many integers $x,y$ such that 
\[\left \lvert y^2 - x^3 \right \vert \sim C x^{\frac{1}{2}}, C = 5^{-5/2} 54 < 1.\]
Therefore, provided that $\G$ is non-constant and negative definite in this region, we will have $\inf_{(x,y) \in \bZ^2} F(x,y) = -\infty$ and hence cannot be arithmetically positive. This completes the proof. 
\end{proof}

\end{document}